\documentclass[11 pt]{article}
\usepackage{geometry}
\usepackage{amsmath}
\usepackage{amsfonts}
\usepackage{amssymb}

\DeclareMathOperator{\spn}{span}
\DeclareMathOperator{\sign}{sign}

\newtheorem{theorem}{Theorem}[section]
\newtheorem{corollary}{Corollary}[section]
\newtheorem{lemma}{Lemma}[section]
\newtheorem{remark}{Remark}[section]
\newtheorem{prop}{Proposition}[section]

\newcommand{\RR}{\mathbb R}
\newcommand{\A}{\mathcal A}
\newcommand{\B}{\mathcal B}
\newcommand{\D}{\mathcal D}
\newcommand{\ZZ}{\mathcal Z}
\newcommand{\const}{\operatorname{const}}

\newcommand{\qed}{{}\hfill $\square$ \\}
\newenvironment{proof}{{\bf Proof.}}{\qed}

\begin{document}

\title{Geometric aspects of two- and threepeakons}
\author{Tomasz Cie\'slak\footnote{cieslak@impan.pl}\\
{\small Institute of Mathematics, Polish Academy of Sciences,}\\
{\small Warszawa, Poland}
\and
Wojciech Kry\'nski,\footnote{krynski@impan.pl}\\
{\small Institute of Mathematics, Polish Academy of Sciences,}\\
{\small Warszawa, Poland}}

\date{}
\maketitle
\begin{abstract}
We apply geometric tools to study dynamics of two- and threepeakon solutions of the Camassa--Holm equation.
New proofs of asymptotic behavior of the solutions are given. In particular we recover well-known collision conditions.
Additionally the Gauss curvature (in the twopeakon case) and the sectional curvature (in the treepeakon case) of corresponding manifolds are computed.
\end{abstract}


\section{Introduction}\label{sec:intro}
This paper is devoted to the study of multipeakons which are particular solutions of the Camassa--Holm equation
\begin{equation}\label{eq:CH}
u_t-u_{xxt}+3uu_x-2u_xu_{xx}-uu_{xxx}=0.
\end{equation}
The equation was introduced in \cite{CH,CHH} as a model for waves in shallow water and it plays a fundamental role in the theory of integrable equations (see \cite{BSS,C1,C2,CEsch,KM}).

A multipeakon, or an $n$-peakon, is a function of the following form
\begin{equation}\label{eq:multi}
u(x,t)=\sum_{i=1}^{n}p_i(t)e^{-|x-q_i(t)|}.
\end{equation}
It is a weak solution of the Camassa--Holm equation provided $p_i(t)$ and $q_i(t)$ evolve accordingly to the Hamiltonian system
\begin{equation}\label{eq:hamilton}
\dot{q_i}=\frac{\partial H}{\partial p_i}, \qquad \dot{p_i}=-\frac{\partial H}{\partial q_i}, \qquad i=1,\ldots,n,
\end{equation}
where $p=(p_1,\ldots,p_n)$, $q=(q_1,\ldots,q_n)$ and the Hamiltonian function $H$ is given by the following formula
\[
H(p,q)=\frac{1}{2}\sum_{i,j=1}^np_ip_je^{-|q_i-q_j|}.
\]
Therefore there is a one to one correspondence between the $n$-peakon solutions of \eqref{eq:CH} and solutions of \eqref{eq:hamilton}.

Notice that the Hamiltonian is a quadratic form that can be written as $H(p,q)=\frac{1}{2}\langle E(q)p,p\rangle$,   where $E=(E_{ij})$ is a symmetric matrix with $q$-dependent entries $E_{ij}(q)=e^{-|q_i-q_j|}$. One can prove that matrix $E(q)$ is positive definite for all $q$ satisfying $q_i\neq q_j$ for $i\neq j$. Consequently, the field of the inverse matrices $q\mapsto g(q)$, where
\[
g(q)=E(q)^{-1},
\]
defines a Riemannian metric on an open subset of $\RR^n$  (the hyperplanes $q_i= q_j$, $i\neq j$ can be considered as singular points of the metric). Further, $H$ can be interpreted as the Hamiltonian function of metric $g$ and, as a result, there is a one to one correspondence between $n$-peakons and geodesics of the metric. This geometric viewpoint has been already introduced in \cite{CH} and provides a very convenient framework to study multipeakons. However, most of the papers concentrate on the analytic properties of the Hamiltonian system \eqref{eq:hamilton} neglecting the underlying geometry. Our goal in this paper is to fill this gap and apply geometric methods to study asymptotic behavior of the multipeakons. It is a continuation of our previous works \cite{CGKM} and \cite{K}. In particular we refer to \cite{CGKM} for an explicit formula for $g$.

A multipeakon $u(x,t)$ given by formula \eqref{eq:multi} collides at time $t^*>0$ if $q_i(t^*)=q_j(t^*)$ for some distinct $i$ and $j$. If it is the case then $t^*$ is referred to as a collision time for $u$ and $q^*=q(t^*)$ is referred to as the corresponding collision point. Note that at a collision time an $n$-peakon becomes an $(n-1)$-peakon. From the geometric viewpoint a collision occurs when a geodesic of $g$ hits one of the hyperplanes $q_i=q_j$, i.e. it approaches a singular point of metric $g$.

The studies on the dynamics of multipeakons were initiated already in \cite{CH}, while in \cite{BSS} very complex results concerning the collisions were given.  In particular it is proved that a multipeakon \eqref{eq:multi} collides in a finite time if and only if there exists $i\in \{1,\ldots,n-1\}$ such that $q_i<q_{i+1}$ and at the same time $p_i>0$ and $p_{i+1}<0$. Similar results, using different methods, were provided in \cite{HR1, HR2}. Moreover, in \cite{GH}  very precise results concerning asymptotics of a twopeakon are stated. The problem of a prolongation of a solution after a collision time has been studied by many authors. The results split into two main branches depending on the class of solutions: dissipative or conservative (see \cite{BC,EG,HR1,HR2}).

The present paper carries on studies of multipeakons exploiting the geometric approach of \cite{CGKM,K}. We concentrate on two- and threepeakons and provide new proofs of the upper mentioned result concerning collisions (and their lack). Our idea is to utilize solely geometric properties of the system  and to exploit tools of differential geometry. The tools are briefly described in the following Section \ref{sec:pre}. Then, in Sections \ref{sec:2D} and \ref{sec:3D}, respectively, we analyze the two- and threepeakons. Moreover, we present new quantitative estimates for twopeakons basing on first integrals of the system. This is the content of Section \ref{sec:2Das}.

Additional outcome of our study is an explicit formula for the curvature of $g$ (the Gauss curvature in the case of twopeakons and the sectional curvature in the case of threepeakons). The curvature is used in the present paper to analyze asymptotic behavior of geodesics of $g$. A surprising phenomenon occurs: the curvature is not of a constant sign (which makes the analysis more difficult). On the other hand we prove it is bounded in a neighborhood of the singular set of $g$ and decays to zero at infinity (which reflects the fact that $g$ is close to the standard Euclidean metric at infinity).

\section{Preliminaries}\label{sec:pre}

In this section we shall recall basic properties of the Hamiltonian system \eqref{eq:hamilton} and metric $g$. We start by fixing our notation.

Recall that we are dealing with $\RR^n$ equipped with linear coordinates $(q_1,\ldots,q_n)$ and metric $g$. In the coordinate system $g$ is written as
$$
g=\sum_{i,j=1}^ng_{ij}dq_idq_j
$$
where $(g_{ij})$ is the inverse of $(E_{ij})$ with $E_{ij}=e^{-|q_i-q_j|}$. In above, $dq_i$'s are one-forms dual to the vector fields $\partial_i=\frac{\partial}{\partial q_i}$. The tuple of one-forms $(dq_1,\ldots,dq_n)$ constitute a coframe on $\RR^n$. Thus, any covector $\alpha$ in the cotangent space $T^*_q\RR^n$ is written as $\alpha=p_1dq_1+\cdots+p_ndq_n$, for some coefficients $(p_1,\ldots,p_n)$. In this way the standard vertical coordinate functions $p=(p_1,\ldots,p_n)$ on $T^*\RR^n$ are introduced. Note that the condition $p_i=0$ translates to $\alpha(\partial_i)=0$.

The singular set of $g$, consisting of the hyperplanes $q_i=q_j$, $i\neq j$, will be denoted $\Sigma$, i.e.
\[
\Sigma=\{q=(q_1,\ldots,q_n)\in\RR^n\ \colon\ \exists_{i,j}\ q_i=q_j,\ i\neq j\}.
\]
As pointed out in the Introduction, the points in $\RR^n\setminus\Sigma$ are regular points of $g$. Note that this domain is not connected. However, due to the invariance of $H$ with respect to the order of coordinates, without loss of generality, one can study the dynamics of the system within the subset
\[
\Omega=\{q=(q_1,\ldots,q_n)\in\RR^n\ \colon\ q_1>q_2>\cdots>q_n\}.
\]
Let us recall that from the point of view of the Camassa--Holm equation, coordinate functions $q_i$'s represent positions of subsequent peaks of a multipeakon. Limiting to $\Omega$ means that we index them from the right to the left.

\subsection{Killing vector field}\label{sec:killing}
A Killing vector field of a metric $g$ is an infinitesimal symmetry of $g$, i.e. a vector field $X$ satisfying $L_Xg=0$ where $L_X$ is the Lie derivative in the direction of $X$. In the present case it is clear that the Hamiltonian, and consequently system \eqref{eq:hamilton}, is invariant with respect to the translations of the form
\[
q_i\mapsto q_i+c,\qquad i=1,\ldots,n.
\]
It follows that
\[
X=\partial_1+\cdots+\partial_n
\]
is a Killing vector field of the associated metric $g$. Because of that the product $g(\dot\gamma,X)$ is preserved along $\gamma$, provided $\gamma$ is a geodesic curve of $g$.

\subsection{First integrals}\label{sec:integrals}
System \eqref{eq:hamilton} is a Liouville integrable system in $\Omega$ and as an integrable system it possesses $n$ independent first integrals $H_0, H_1,\ldots,H_{n-1}$ (see \cite{K} for a detailed exposition). It appears that $H_i$ is a homogeneous polynomial in $p=(p_i)$ of degree $i+1$. In particular
\[
H_0(q,p)=p_1+\cdots+p_n
\]
is linear in $p$. It is worth to notice that it is a conserved quantity associated to the upper mentioned Killing vector field $X$ of metric $g$. Indeed, one observes that along a geodesic $\gamma$ one gets $H_0=g(\dot\gamma,X)$ (compare \eqref{eq:pq} below).

Further, $H_1$ is quadratic in $p$ and, as one expects, it is the original Hamiltonian $H$
\[
H_1(q,p)=\frac{1}{2}\sum_{i,j=1}^np_ip_je^{-|q_i-q_j|}.
\]
The third first integral $H_2$ is cubic in $p$. For $n=3$ it takes the following form
\[
H_2(p,q)=\frac{1}{3}(p_1^2+p_2^2+p_3^2+3e^{(q_1-q_2)}(p_1^2p_2+p_1p_2^2)+3e^{(q_2-q_3)}(p_2^2p_3+p_2p_3^2)
\]
\[
+3e^{(q_1-q_3)}(p_1^2p_3+p_1p_3^2)+6e^{-(q_1-q_3)}p_1p_2p_3).
\]
The higher first integrals can be explicitly written down using the bi-Hamiltonian approach of \cite{K}. However, we shall not need them in the present paper and refer an interested reader to \cite{K}.

\subsection{Invariant sets}\label{sec:invset}
The flow of \eqref{eq:hamilton} preserves the level sets of $H_i$'s (which are the subsets of $T^*\RR^n\simeq\RR^{2n}$). Unfortunately, since the degree of $H_i$ grows with $i$, it is a very hard task to find the level sets of the Hamiltonians for $n>2$. However, one can easily observe the following.
\begin{prop}\label{prop:zeroset}
The set
\[
\ZZ=\bigcup_{i=1}^n\{(q,p)\in\RR^{2n}\ \colon\ p_i=0\}
\]
is an invariant subset of the Hamiltonian flow of \eqref{eq:hamilton}.
\end{prop}
\begin{proof}
Follows from \eqref{eq:hamilton} which reads $\dot p_i=\frac{1}{2}p_i\sum_{j\neq i}p_je^{-|q_i-q_j|}\sign(q_i-q_j)$, i.e. if $p_i=0$ then $\dot p_i=0$.
\end{proof}

From the geometric viewpoint, $\ZZ$ is a subset of the cotangent bundle $T^*\RR^n$ (which is naturally identified with $\RR^{2n}$). In the forthcoming sections it will be convenient to pass to the tangent bundle. For this we exploit the duality between the tangent and cotangent bundles established by $g$. Namely, if $X\in T_q\RR^n$ is a tangent vector then the dual covector $X^*$ is defined by the formula
$$
X^*(Y)=g(X,Y), \ \ \mbox{for all}\ \ Y\in T_q\RR^n.
$$
Applying the duality to conditions $p_i=0$, $i=1,\ldots,n$, at any $q\in\Omega$ one defines $n$ codimension 1 subspaces of $T_q\RR^n$. Indeed,
\[
\D_i(q)=\{X\in T_q\RR^n\; :\; X^*(\partial_i)=0\}.
\]
In this way we get $n$ vector distributions $q\mapsto\D_i(q)$, $i=1,\ldots,n$, each of rank $n-1$, defined at all points of the considered domain $\Omega$ in $\RR^n$. Later on we shall investigate properties of $\D_i$'s in details. At this point we shall only recall that the duality between $T\RR^n$ and $T^*\RR^n$ can be also seen on the level of solutions to \eqref{eq:hamilton}. Indeed, any solution $t\mapsto(q(t),p(t))$ to the Hamiltonian system \eqref{eq:hamilton} satisfies
\begin{equation}\label{eq:pq}
p=E^{-1}\dot{q}=g(\dot{q},.),
\end{equation}
where $E$ is the dual matrix of $g$. It follows that the curve $t\mapsto q(t)$ itself determines the corresponding second factor $p(t)$ uniquely. Consequently, the solutions of \eqref{eq:hamilton} contained in the invariant set $\ZZ$ are lifts of geodesics of $g$ that are tangent to distributions $\D_i$'s.

To complete this section let us notice that in the low dimensional cases $n=2$ and $n=3$ one can come to the conclusion of Proposition \ref{prop:zeroset} in a way that does not involve \eqref{eq:hamilton}. Namely, one can consider
\[
\hat H=H_1-\frac{1}{2}H_0^2
\]
for $n=2$ or
\[
\hat H=H_2-H_0H_1+\frac{1}{6}H_0^3
\]
for $n=3$. Then, direct computations prove in both cases that the corresponding  zero sets of $\hat H$ consist of $\ZZ$ and $\Sigma$. Indeed
\[
\hat H=p_1p_2\left(e^{-|q_1-q_2|}-1\right)\;\;\mbox{for}\;\;n=2
 \]
and
\[
\hat H=p_1p_2p_3\left(1+e^{-(q_1-q_3)}-e^{-(q_1-q_2)}-e^{-(q_2-q_3)}\right)\;\;\mbox{for}\;\; n=3.
\]
In this way one recovers $\ZZ$ purely in terms of the first integrals.

\paragraph*{Remark.}
In order to justify the definition of $\hat H$ for $n=2$ notice that it is a quadratic function with respect to $p$. One can consider more general function $\hat H_\epsilon=H_1-\epsilon H_0^2$, so that $\hat H$ defined above corresponds to $\epsilon=\frac{1}{2}$. For all values of $\epsilon$ it is a conserved quantity of \eqref{eq:hamilton}. Further, similarly to the case $\epsilon=0$, $\hat H_\epsilon$ corresponds to a certain (pseudo-)Riemannian metric $g_\epsilon$ on an open subset of $\RR^2\setminus\Sigma$. If $\epsilon$ is negative then the metric is positive definite. On the other hand, it turns out that $\epsilon=\frac{1}{2}$ is the smallest $\epsilon$ such that $g_\epsilon$ has Lorentzian signature on the whole set $\RR^2\setminus\Sigma$.

\subsection{Comparison Theorems}\label{sec:comparison}
A fundamental tool in our approach will be the Rauch comparison theorem that give estimates for the behavior of geodesics in terms of the sectional curvature. We recommend a classic book \cite{CE} for details on the subject. Here, in order to fix the notation we recall that for a given point $q$ in a manifold $M$ and a 2-dimensional subspace $\sigma$ of the tangent space $T_qM$ the sectional curvature $\kappa_\sigma$ can be defined as the Gauss curvature of the 2-dimensional geodesic submanifold of $M$ tangent to $\sigma$ at $q$. Later on, in Section \ref{sec:curvature3D} we shall provide explicit formulae for the sectional curvature in terms of the components of the Riemann tensor (in dimension 2 one can consider just the Gauss curvature instead of the sectional curvature).

The following result  follows from \cite[Theorem 1.28]{CE} with $M_0$ being a manifold of constant sectional curvature $\kappa$ (see a discussion following \cite[Theorem 1.28, page 30]{CE}).
\begin{theorem}[Rauch]\label{thm:rauch1}
Let $M$ be a Riemannian manifold such that for any 2-dimensional subspace $\sigma$ of the tangent bundle $TM$
$$
\kappa_\sigma<\kappa
$$
for some constant $\kappa>0$. Then for any normal geodesic $\gamma\colon[0,T]\to M$ its first conjugate time is no earlier than at time $t^*=\frac{\pi}{\sqrt{\kappa}}$. In particular, if $\kappa_\sigma<0$ then there are no conjugate points on $\gamma$.
\end{theorem}
In above, a conjugate time for a  normal geodesic $\gamma\colon[0,T]\to M$ is a time $t^*<T$ such that there is a nontrivial Jacobi vector field $J$ along $\gamma$ satisfying $J(0)=J(t^*)=0$, where a Jacobi vector field is defined as an infinitesimal family of geodesics deforming $\gamma$, and a geodesic is normal if $|\dot\gamma(t)|=1$ (see e.g. \cite{CE}). In particular, if $\gamma(0)$ and $\gamma(t^*)$ are connected by two different geodesics (in a simply connected domain) then $t^*$ is a conjugate time for $\gamma$. We shall also use (in the 2-dimensional case only) the following result, witch follows \cite[Corollary 1.30]{CE} with $c$ being a geodesic in $M$ and $M_0$ being a flat Euclidian space.
\begin{corollary}[Rauch]\label{thm:rauch2}
If $M$ has a negative sectional curvature then any two geodesics emerging from a point $q\in M$ diverge at least as fast as straight lines in the Euclidian space.
\end{corollary}

\subsection{sub-Riemannian corank-1 structures}\label{sec:subR}
A general sub-Riemannian structure is given by a pair $(\D,h)$ where $\D$ is a vector distribution on a manifold and $h$ is a metric (a bi-linear, positive-definite product) on $\D$. It is usually assumed that the dimension of $\D(q)$ is independent of $q$ and this dimension is referred to as the rank of $\D$. It follows  that locally, around any point $q$, a distribution of rank $k$ is spanned by $k$ point-wise independent vector fields $X_1,\ldots,X_k$, i.e. $\D(q)=\spn\{X_1(q),\ldots,X_k(q)\}$.

A horizontal curve of $\D$ is a curve $\gamma$ (of appropriate regularity) that is tangent to $\D$ a.e., meaning that
\[
\dot\gamma(t)\in\D(\gamma(t))\qquad\mbox{a.e.}
\]
Note that metric $h$ can be used to define a length of a horizontal curve. A horizontal curve is called a sub-Riemannian geodesic if it is length-minimizing among all horizontal curves joining two given points.

Distribution $\D$ locally spanned by $X_1,\ldots,X_k$ is called involutive (or integrable) if all Lie brackets $[X_i,X_j]$ of vector fields  spanning $\D$ are sections of the original distribution $\D$ (c.f. the Frobenius theorem). Clearly, this notion does not depend on the choice of the vector fields spanning $\D$. On contrary, a distribution is called totally non-holonomic if all iterated Lie brackets of $X_1,\ldots,X_k$ span the whole tangent space (a number of iteration may vary from point to point). The fundamental Chow--Rashevskii theorem says that if a distribution $\D$ is totally non-holonomic then any two points in a connected component of the underlying manifold can be connected by a horizontal curve of $\D$.

In our case we get $n$ sub-Riemannian structures on $\Omega$ given by $\D_i$ with metrics $h_i$ defined as restrictions of $g$ to $\D_i$. All $\D_i$'s are of rank $n-1$ (i.e. they are of corank one). That means that all $\D_i$'s are (at least locally) defined as kernels of a one-form, say $\alpha_i$, which is given up to a multiplicative factor
\[
\ker\alpha_i=\D_i.
\]
The involutivity of $\D_i$ can be checked in terms of $\alpha_i$ instead of the Lie brackets. We shall use the following.
\begin{prop}\label{prop:contact}
A corank 1 distribution on a 3 dimensional manifold defined as a kernel of a one-form $\alpha$ is involutive if and only if
\[
d\alpha\wedge\alpha=0.
\]
\end{prop}
In the following sections it will be of fundamental importance for understanding of the geometry of multipeakons to determine which distributions among $\D_i$'s are integrable. Note that in dimension 2 all corank-1 distributions are integrable (they are spanned by a vector field). On contrary, in higher dimensions a generic corank-1 distribution is non-holonomic. This phenomenon is reflected in greater complexity of the problem in higher dimensions.

\section{Dynamics of twopeakons}\label{sec:2D}
In this section we give another proof of the sufficient and necessary conditions for the collisions of twopeakons. For earlier proofs we refer to \cite{BSS, CGKM, K, GH, HR1, HR2}. Our new proof seems to be the easiest one.  Moreover we shall use an extension of a two-dimensional approach in the geometrically more complicated 3D case, which we deal with in Section~\ref{sec:3D}. Additionally, the asymptotics of twopeakons that do not collide is studied at the end of this section.

In the two-dimensional case metric $g$ in domain $\Omega$ is given by the following explicit formula
\[
g_{ij}=\frac{1}{\left(1-e^{-2(q_1-q_2)}\right)}(-1)^{i+j}e^{-|q_i-q_j|}.
\]
We start with the following Lemma, which is also interesting on its own as the curvature is not of constant sign.
\begin{lemma}\label{lemma:GC}
The Gauss curvature $\kappa_g$ of metric $g$ in domain $\Omega$ satisfies
\begin{equation}\label{eq:kappa}
\kappa_g=\frac{e^{(q_1-q_2)}-2}{e^{2(q_1-q_2)}+2e^{(q_1-q_2)}+1}.
\end{equation}
In particular,
\[
\kappa_g>0\;\mbox{if}\;q_1-q_2>\ln 2,\;\kappa_g=0\;\mbox{if}\;q_1-q_2=\ln 2,\;\mbox{and}\;\kappa_g<0\;\mbox{if}\;q_1-q_2<\ln 2.
\]
\end{lemma}
\begin{proof}
First, as in \cite[Theorem 3.1]{CGKM}, we introduce new variables
\[
s_1:=\frac{q_1+q_2}{2}, \;\;\; s_2:=\frac{q_1-q_2}{2}\;.
\]
Next, we notice that in this coordinates the metric takes a diagonal form
\[
g=\left[
	\begin{array}{cc}
	\frac{2}{1+e^{-2s_2}} & 0\\
	0 & \frac{2}{1-e^{-2s_2}}
	\end{array}
\right]\;.
\]
The advantage of new variables is that the Christoffel symbols are easy to compute. Let us recall that
\[
\Gamma_{ij}^k=1/2\sum_{r=1}^2\left(\frac{\partial g_{ir}}{\partial s_j}+\frac{\partial g_{jr}}{\partial s_i}-\frac{\partial g_{ij}}{\partial s_r}\right)g^{kr},
\]
where $g^{kr}$ are the $kr$ entries of the inverse of $g$.

We obtain
\[
\Gamma_{11}^1=\Gamma_{22}^1=\Gamma_{12}^2=0.
\]
and
\[
\Gamma_{12}^1=\frac{e^{-2s_2}}{1+e^{-2s_2}}\;,
\Gamma_{11}^2=-\frac{e^{-2s_2}\left(1-e^{-2s_2}\right)}{\left(1+e^{-2s_2}\right)^2}\;,\;\;
\Gamma_{22}^2=-\frac{e^{-2s_2}}{1-e^{-2s_2}}.
\]
Next, we recall that the Gauss curvature is expressed with the use of Christoffel's symbols as
\[
-g_{11}\kappa_g=\left(\frac{\partial \Gamma_{12}^2}{\partial s_1}-\frac{\partial \Gamma_{11}^2}{\partial s_2}+\Gamma_{12}^1\Gamma_{11}^2+\Gamma_{12}^2\Gamma_{12}^2-\Gamma_{11}^1\Gamma_{12}^2-\Gamma_{11}^2\Gamma_{22}^2\right),
\]
so that in our case
\[
-g_{11}\kappa_g=\left(-\frac{\partial \Gamma_{11}^2}{\partial s_2}-\frac{e^{-4s_2}\left(1-e^{-2s_2}\right)}{\left(1+e^{-2s_2}\right)^3}-\frac{e^{-4s_2}}{\left(1+e^{-2s_2}\right)^2}\right).
\]
And consequently
\begin{eqnarray*}
&&-g_{11}\kappa_g=\frac{6e^{-4s_2}-2e^{-2s_2}}{\left(1+e^{-2s_2}\right)^3}-\frac{e^{-4s_2}\left(1-e^{-2s_2}\right)}{\left(1+e^{-2s_2}\right)^3}-\frac{e^{-4s_2}}{\left(1+e^{-2s_2}\right)^2}\\
&=&\frac{6e^{-4s_2}-2e^{-2s_2}-e^{-4s_2}+e^{-6s_2}-e^{-4s_2}\left(1+e^{-2s_2}\right)}{\left(1+e^{-2s_2}\right)^3}\\
&=&\frac{4e^{-4s_2}-2e^{-2s_2}}{\left(1+e^{-2s_2}\right)^3}=\frac{2e^{-2s_2}\left(2e^{-2s_2}-1\right)}{\left(1+e^{-2s_2}\right)^3}\;.
\end{eqnarray*}
Hence
\[
\kappa_g=\frac{e^{-4s_2}\left(e^{2s_2}-2\right)}{e^{-4s_2}\left(e^{2s_2}+1\right)^2}=\frac{e^{2s_2}-2}{\left(e^{2s_2}+1\right)^2}\;.
\]
\end{proof}
Tracing the estimates of Lemma \ref{lemma:GC} also in the upper half-plane $\{(q_1,q_2):q_1<q_2\}$, we notice that
\[
\kappa_g=\frac{e^{q_2-q_1}-2}{\left(e^{q_2-q_1}+1\right)^2}
\]
there and arrive therefore at the following remark concerning the singularity of the Gauss curvature of $g$.
\begin{remark}
The Gauss curvature $\kappa_g$ of a twopeakon metric $g$ satisfies
\[
\lim_{q_1-q_2\rightarrow 0}\kappa_g(q_1,q_2)=-\frac{1}{4}\;.
\]
\end{remark}

Now we shall study distributions $\D_1$ and $\D_2$ introduced in Section \ref{sec:invset}. In the present case, both distributions are of rank 1. Therefore, they are integrable and one can consider the corresponding integral curves instead of the distributions. The set of (unparameterized) curves tangent to $\D_1$ will be denoted $\A$ and the set of curves tangent to $\D_2$ will be denoted $\B$.

\begin{prop}\label{prop:2Dfol}
There exist exactly one integral curve belonging to $\A$ and exactly one integral curve belonging to $\B$ passing through a given point $q=(q_1,q_2)\in\Omega$. All curves in $\A$ approach asymptotically the singular set $\Sigma$ for $q_1\rightarrow\infty$,  while $q_1-q_2\rightarrow\infty$ for $q_1\rightarrow -\infty$. Similarly, all curves in $\B$ approach set $\Sigma$ for $q_1\rightarrow -\infty$ and $q_1-q_2\rightarrow\infty$ for $q_1\rightarrow \infty$.
Both families $\A$ and $\B$ constitute foliations of $\Omega$. Moreover, any curve in $\A$ is transversal to any curve in $\B$.
\end{prop}
\begin{proof}
From \eqref{eq:pq} we see that the condition $p_1=0$ is equivalent to
\[
\dot{q_1}-e^{-(q_1-q_2)}\dot{q_2}=0,
\]
i.e.
\[
\frac{d}{dt}\left(e^{q_1}-e^{q_2}\right)=0\;\;\mbox{which gives}\;\;e^{q_1}-e^{q_2}=\const,
\]
which is an equation for the foliation defined by $\A$. The same computation gives $e^{-q_1}-e^{-q_2}=\const$ as an equation of the foliation defined by $\B$.
Further, it follows that, for any point $q$ in the halfplane $q_1>q_2$ there exists exactly one curve in the family $\A$ passing throuh $q$ and, similarly, exactly one curve in the family $\B$ passing through $q$.
Corresponding curves are transversal.

Now, let us find asymptotics of $\A$ and $\B$. We start with $\A$ and we have
\[
\frac{d}{dt}(q_1-q_2)=p_2(0)\left(e^{q_2-q_1}-1\right),
\]
so that $z:=q_1-q_2$ satisfies $\dot{z}=p_2(0)\left(e^{-z}-1\right)$. It is clear then that
if $p_2(0)>0$, then $z(t)\rightarrow 0$ for $t\rightarrow \infty$ and $z(t)\rightarrow \infty$ for
time going back to $-\infty$. For $p_2(0)<0$ the situation is opposite. The same computation
shows asymptotics for curves of $\B$ family.
\end{proof}

Recall that by Section \ref{sec:invset} the curves in $\A$ and $\B$ are geodesics of $g$.  They will play a crucial role in the following proof of 2-dimensional version of necessary and sufficient condition for collisions.
\begin{theorem}\label{thm:collisions2D}
Let $u(x,t)=p_1(t)e^{-|x-q_1(t)|}+p_2(t)e^{-|x-q_2(t)|}$ be a twopeakon solution to the Camassa-Holm equation with initial data $(q(0),p(0))$ satisfying $q(0)\in\Omega$. Then the twopeakon collides in a finite time if and only if
\begin{equation}\label{eq:cond2D}
p_2(0)>0>p_1(0).
\end{equation}
\end{theorem}

\begin{proof}
According to Proposition \ref{prop:2Dfol}, the two special geodesics of $g$ from families $\A$ and $\B$, respectively, that pass through the point $q(0)=(q_1(0), q_2(0))$, intersect transversally and approach asymptotically the singular set $\Sigma$, which is the boundary of $\Omega$. It follows that the halfplane $\Omega$ is divided into four sectors.
Sector I is located between the line $q_1=q_2$ and parts of curves from $\A$ and $\B$ emerging from the original point $(q_1(0),q_2(0))$ and approaching the line $q_1=q_2$ in the infinity. Sector II (resp. III) is located between parts of $\A$ and $\B$ to the right (resp. left) from the point $(q_1(0),q_2(0))$. Finally, sector IV is located between parts of curves $\A$ and $\B$ emerging from $(q_1(0),q_2(0))$ and moving away from the line $q_1=q_2$.

We claim that any solution starting at $(q_1(0),q_2(0))$ and directed into one of the Sectors I, II, III, IV, stays there. In particular, solutions from Sector II, III or IV never approach $\Sigma$. Indeed, we shall prove that they are bounded away from it by geodesics from families $\A$ and $\B$ respectively. In the last part of the proof we shall show that any trajectory hitting initially Sector I attains the set $q_1=q_2$ (which is equivalent to the collision of a twopeakon) at a finite time.

In order to show that a given solution does not leave its initial sector for any $t>0$, we shall exploit Proposition \ref{prop:2Dfol}. Recall that function $\hat H$ from Section \ref{sec:invset} is a constant of motion. Hence the sign of the product $p_1p_2$ is also a constant of motion. Consequently, none of $p_i$ can become $0$ along a solution and both $p_i$, $i=1,2$, preserve signs during the motion. On the other hand, due to the duality \eqref{eq:pq} between $p$ and $\dot q$, we can assign signs of $p_1(0)$ and $p_2(0)$ to geodesics emerging from $q(0)$ in directions belonging to different sectors (note that \eqref{eq:cond2D} corresponds to Sector I). Now,  assume that there exists a time, say $t_1$, such that a given solution hits the boundary of its sector at time $t_1$ i.e.\;the solution curve intersects either the curve in the family $\A$ or the curve in the family $\B$ originating from the initial point $q(0)$. Without loss of generality, we assume that it is the curve in $\A$. Then, it follows that the solution curve intersects the same curve in $\A$ twice: at $t=0$ and at $t=t_1$. We can repeat a construction of four sectors at point $q(t_1)$ and we get to the conclusion that the curve emerges from $q(t_1)$ into a different sector. Consequently $p_1$ changes its sign along the curve, which is a contradiction.

We have proved that any solution in Sectors II, III and IV does not collide. Let us show that any solution curve $\gamma$ in Sector I gives a finite time collision. First observe that since the product $g(\dot\gamma,X)$ is constant for $X$ being the Killing vector field, $\gamma$ decreases the euclidean distance to $\Sigma$ with time. Moreover, Sector I is bounded by curves in families $\A$ and $\B$ (which approach $\Sigma$ at infinity). We thus conclude that the curve $\gamma$ either approaches the singular set $\Sigma$ at infinity or there is a finite time collision. We shall exclude the first possibility. Assume the converse. Since $\gamma$ approaches  $\Sigma$,  we can assume that it is contained in the region $\Omega_-$ of $\Omega$,  for which the Gauss curvature is negative (see Lemma \ref{lemma:GC}). In this region we apply Corollary \ref{thm:rauch2} and get that the Euclidean distance between $\gamma$ and a curve in family $\A$ (or $\B$) grows to infinity. One gets a contradiction. Consequently $\gamma$ necessarily hits $\Sigma$ at a finite time.
\end{proof}

\subsection{Asymptotic estimates for twopeakons}\label{sec:2Das}
In this section we shall provide explicit quantitative estimates for collisions (and their lack) of twopeakons. For this we use the first integrals of Section \ref{sec:integrals}.

\begin{theorem}\label{thm:estimmates2D}
Let $u(x,t)=p_1(t)e^{-|x-q_1(t)|}+p_2(t)e^{-|x-q_2(t)|}$ be a twopeakon solution to the Camassa-Holm equation with initial data $(q(0),p(0))$ satisfying $q(0)\in\Omega$. Then there is a collision not later than at time
\begin{equation}\label{eq:colision2D}
t^*=\frac{2\sqrt{1-y(0)}}{y(0)\sqrt{(1+y(0))(2H_1-H_0^2)}},
\end{equation}
where $y(0)=e^{q_2(0)-q_1(0)}$, or
\begin{equation}\label{eq:escape2D}
q_1(t)-q_2(t)\rightarrow \infty\;\;\mbox{when}\;\;t\rightarrow\infty .
\end{equation}
\end{theorem}
\begin{proof}
Let us define $z(t):=q_1(t)-q_2(t)$. First we assume that there is a collision. Then, as in the proof of Theorem \ref{thm:collisions2D} the geodesic $t\mapsto q(t)$ is in Sector I.  We notice that in Sector~I $\hat H<0$, meaning that $2H_1-H_0^2>0$, where $\hat H$ is a constant of motion defined in Section \ref{sec:invset} (recall that the boundaries of sectors are defined by the equation $\hat H=0$). Moreover,
\begin{equation}\label{eq:ss}
\dot{z}=(1-e^{-z})(p_1-p_2).
\end{equation}
Observe that
\[
p_1-p_2=-\sqrt{2H_1-2p_1p_2(1+e^{-z})}.
\]
Hence
\begin{eqnarray}\label{eq:chic}
\dot{z}&=&-(1-e^{-z})\sqrt{2H_1-2p_1p_2(1+e^{-z})}\nonumber \\
&=&-\sqrt{(1-e^{-z})^2\left(2H_1-(1+e^{-z})\frac{H_0^2-2H_1}{1-e^{-z}}\right)}\nonumber \\
&=&-\sqrt{(1-e^{-z})\left(2H_1(1-e^{-z})-(1+e^{-z})(H_0^2-2H_1)\right)}\\
&\leq&-\sqrt{1-e^{-z}}\sqrt{-(1+e^{-z})(H_0^2-2H_1)}.\nonumber
\end{eqnarray}
Substituting $y:=e^{-z}$, we obtain from \eqref{eq:chic}
\begin{equation}\label{eq:imp}
\dot{y}=-e^{-z}\dot{z}\geq -y\left(-\sqrt{1-y}\sqrt{-(1+y)(H_0^2-2H_1)}\right).
\end{equation}
At this stage we notice that $\dot{y}>0$. Indeed, we are in Sector I, so that $p_1(0)<0<p_2(0)$. Signs of $p_i$, $i=1,2$, are preserved by the evolution and so $p_1(t)<0<p_2(t)$ for any $t>0$. Thus $\dot{z}<0$ due to \eqref{eq:ss} and so $\dot{y}>0$.

Hence, $y(t)>y(0)$ for any $t>0$ and \eqref{eq:imp} can be rewritten as
\[
\dot{y}\geq y(0)\sqrt{1-y}\sqrt{(1+y(0))(2H_1-H_0^2)}.
\]
We integrate the latter inequality and arrive at
\[
2\sqrt{1-y(t)}\leq 2\sqrt{1-y(0)}-ty(0)\sqrt{(1+y(0))(2H_1-H_0^2)},
\]
so that $y(t)=1$  (i.e. $z(t)=0$ meaning that a collision takes place) not later than at
\[
\frac{2\sqrt{1-y(0)}}{y(0)\sqrt{(1+y(0))(2H_1-H_0^2)}}.
\]

Now, we consider a twopeakon that does not collide, i.e. it does not satisfy \eqref{eq:cond2D}. Once again we use $z:=q_1-q_2$.
We check that $\frac{d}{dt}\left(p_1-p_2\right)=-p_1p_2e^{-z}$ which in turn gives
\[
\frac{d}{dt}\left(p_1-p_2\right)=\frac{1}{2}\left(H_0^2-y^2\right)e^{-z}.
\]
Denote $h:=p_1-p_2$. Then we rewrite $H_1$ as (see \cite{CGKM})
\begin{eqnarray}\label{eq:ano}
H_1&=&\frac{1}{4}\left(H_0^2+h^2\right)+\frac{1}{4}\left(H_0^2+h^2\right)e^{-z}\nonumber \\
&=&\frac{1}{4}\left(H_0^2+h^2+(H_0^2-h^2)e^{-z}\right).
\end{eqnarray}
Hence
\[
(H_0^2-h^2)e^{-z}=4H_1-H_0^2-h^2,
\]
so that $4H_1-H_0^2=H_0^2e^{-z}+h^2(1-e^{-z})>0$ and denoting $a^2:=4H_1-H_0^2$, \eqref{eq:ano} turns into
\begin{equation}\label{eq:ano1}
\dot{h}=\frac{1}{2}\left(a^2-h^2\right).
\end{equation}
Moreover, \eqref{eq:ss} can be rewritten as
\[
\dot{z}=h\left(1-e^{-z}\right),
\]
so that $\dot{z}>0$ as long as $h>0$. But due to \eqref{eq:ano1}, as long as $h(0)>a$, $h(t)>a$ for any $t>0$. This means that $z(t)$ grows with time and is unbounded. Moreover, if $h(0)>-a$, then $h$ grows and there exists $t_0$ such that $h(t)>0$ for any $t>t_0$. Thus, $z(t)$ grows for $t>t_0$ and tends to infinity with time. So that the only possibility that $z$ does not grow to infinity with time is when
\begin{equation}\label{eq:cc}
h(0)<0\qquad \mbox{and}\qquad h(0)^2>a^2.
\end{equation}
Our claim is that \eqref{eq:cc} is satisfied only when
\begin{equation}\label{eq:hyp}
p_1(0)p_2(0)<0\qquad\mbox{and}\qquad p_1(0)-p_2(0)<0.
\end{equation}
Notice that \eqref{eq:hyp} is equivalent to $p_2(0)>0>p_1(0)$, but this means that we are dealing with initial condition leading to finite-time collision. Hence, the proof is completed, provided we show that \eqref{eq:cc} implies \eqref{eq:hyp}.

In the last step we examine condition \eqref{eq:cc}. On the one hand it means that $p_1(0)-p_2(0)<0$. On the other hand $h(0)^2>a^2$ yields
\[(p_1(0)-p_2(0))^2>2\left(p_1(0)^2+p_2(0)^2+2e^{-z}p_1(0)p_2(0)\right) -p_1(0)^2-p_2(0)^2-2p_1(0)p_2(0),
\]
which leads us to
\[
0>4e^{-z}p_1(0)p_2(0),
\]
and we see that \eqref{eq:cc} implies $p_1(0)<p_2(0)$ and $p_1(0)p_2(0)<0$.
\end{proof}

\section{Dynamics of threepeakons}\label{sec:3D}

In the present section we use geometric tools to study existence of collisions for the threepeakons.  We have (see \cite[Corollary 2.1]{CGKM})
\[
g=E^{-1}=\left[
	\begin{array}{ccc}
	\frac{1}{1-e^{-2(q_1-q_2)}} & -\frac{e^{-(q_1-q_2)}}{1-e^{-2(q_1-q_2)}} & 0\\
	  -\frac{e^{-(q_1-q_2)}}{1-e^{-2(q_1-q_2)}} & \frac{1-e^{-2(q_1-q_3)}}{(1-e^{-2(q_1-q_2)})(1-e^{-2(q_2-q_3)})}&  -\frac{e^{-(q_2-q_3)}}{1-e^{-2(q_2-q_3)}}\\
	  0 &  -\frac{e^{-(q_2-q_3)}}{1-e^{-2(q_2-q_3)}} &  \frac{1}{1-e^{-2(q_2-q_3)}}
	\end{array}
\right]\;,
\]

We start with properties of the three rank-2 vector distributions $\D_1$, $\D_2$ and $\D_3$ introduced in Section \ref{sec:invset}, defined at all points of the considered domain $\Omega\subset\RR^3$. It turns out that there is a substantial difference between the cases of two- and threepeakons. Namely, in dimension 2 distributions $\D_1$ and $\D_2$ are of rank 1 and because of that they can be replaced by two families of curves, denoted $\A$ and $\B$ respectively. On the other hand, in dimension 3, a generic distribution of rank 2 is not integrable. As a matter of fact this is the case when it comes to $\D_2$.
\begin{prop}\label{prop:3Ddistr}
Distributions $\D_1$ and $\D_3$ are integrable, whereas distribution $\D_2$ is non-integrable in $\Omega$. Any leaf of $\D_1$ or $\D_3$ cuts $\Omega$ into two sectors. Moreover leafs of $\D_1$ asymptotically converge to the plane $q_1=q_2$ as $q_1\to\infty$ and leafs of $\D_3$ asymptotically converge to the plane $q_2=q_3$ as $q_3\to-\infty$.
\end{prop}
\begin{proof}
Recall that $\D_i$ is exactly $p_i=0$, $i=1,2,3$, and then
\[
\begin{aligned}
&\D_1=\spn\{e^{-q_1}\partial_{q_1}+e^{-q_2}\partial_{q_2},\partial_{q_3}\},\\
&\D_2=\spn\{e^{-q_1}\partial_{q_1}+e^{-q_2}\partial_{q_2}+e^{-q_3}\partial_{q_3}, e^{q_1}\partial_{q_1}+e^{q_2}\partial_{q_2}+e^{q_3}\partial_{q_3}\},\\
&\D_3=\spn\{e^{q_2}\partial_{q_2}+e^{q_3}\partial_{q_3},\partial_{q_1}\}.\\
\end{aligned}
\]
It immediatelly follows that $\D_1$ and $\D_3$ are integrable. On the other hand $\D_2$ is annihilated by the one-form
\[
\begin{aligned}
\alpha=&\,\,(e^{-(q_1-q_2)}-e^{-(q_1+q_2-2q_3)})dq_1\\
&-(1-e^{-2(q_1-q_3)})dq_2\\
&+(e^{-(q_2-q_3)}-e^{-(2q_1-q_2-q_3)})dq_3,
\end{aligned}
\]
which satisfies $d\alpha\wedge\alpha\neq 0 $ provided $q\in\Omega$. Hence, by Proposition \ref{prop:contact}, $\D_2$ is not integrable. Note that for $q_1=q_2$ distribution $\D_2$ coincides with $\D_1$ and, similarly, for $q_2=q_3$ it coincides with $\D_3$. Moreover, for $q_1=q_2=q_3$ $\D_2$ degenerates to a line spanned by the Killing vector field $\partial_1+\partial_2+\partial_3$.

Now, integral curves of the vector field $e^{-q_1}\partial_{q_1}+e^{-q_2}\partial_{q_2}$ are given by $e^{q_1}-e^{q_2}=\const$. Indeed, if $\dot q_1=e^{-q_1}$ and $\dot q_2=e^{-q_2}$ then $\dot q_1-e^{-(q_1-q_2)}\dot q_2=0$ which implies $\frac{d}{dt}(e^{q_1}-e^{q_2})=0$. Thus, as in Proposition \ref{prop:2Dfol}, $q_1-q_2\to 0$ for $q_1\to\infty$.  Similarly, integral curves of the vector field $e^{q_2}\partial_{q_2}+e^{q_3}\partial_{q_3}$ satisfy $e^{-q_2}-e^{-q_3}=\const$ and consequently $q_2-q_3\to 0$ for $q_3\to-\infty$.
\end{proof}

In the case of twopeakons, the curves $\A$ and $\B$ split the half space $q_1>q_2$ into 4 sectors. The sectors define obstacles preventing geodesics from hiting the singular set $q_1=q_2$. In the case of threepeakons, $\D_1$, $\D_2$ and $\D_3$ divide each tangent space into 8 sectors (at a generic point). However, since $\D_2$ is non-integrable it is impossible to define sectors on the underlying manifold. The sectors are defined in each tangent space only. Nonetheless, the distribution $\D_2$ can be used to define obstacles for geodesics in a more subtle way. We start with the following result that exploits results of \cite{K} on asymptotic behavior of geodesics in neighbourhoods of $\Sigma$. In fact it is a direct consequence of \cite[Lemma 4.1]{K}.

\begin{lemma}\label{lemma:D2}
Assume that a geodesic $t\mapsto q(t)$ of metric $g$ is a horizontal curve of $\D_2$, i.e. $\dot q(t)\in\D_2(q(t))$. Then, if the geodesic converges in a finite time to a singular point $q^*$ of $g$ then $q^*$ belongs to the line $q_1=q_2=q_3$.
\end{lemma}
\begin{proof}
Since $\D_2$ is non-integrable, then there exist horizontal curves of $\D_2$ that converge to any point in $\Omega$. We shall prove that it is not the case when it comes to geodesics. We proceed by contradiction. For this, let $q^*$ be a point such that $q^*_2=q^*_i$, but $q^*_j\neq q^*_2$, where $\{i,j\}=\{1,3\}$. Then, according to Lemma 4.1 in \cite{K}, $p_2+p_i$ is bounded and $p_2-p_i$ tends to infinity as $q(t)\to q^*$. But it is impossible for $p_2=0$.
\end{proof}

The reasoning of Lemma \ref{lemma:D2} can be applied to $\D_1$ and $\D_3$ as well (although, in some sense stronger properties of $\D_1$ and $\D_3$ have been already described in Proposition \ref{prop:3Ddistr} above). Indeed, note that since $\D_i=\{p_i=0\}$, any geodesic tangent to any $\D_i$, $i=1,2,3$, represents not a threepeakon but a twopeakon as there are only two components left in \eqref{eq:multi}. However, the corresponding geodesic, considered as a curve in $\RR^3$, encodes a position of the third peak of amplitude 0. This third peak evolves in time in some way according to \eqref{eq:multi}. No matter this evolution is, Lemma \ref{lemma:D2} says that it cannot collide with any other peak alone. Proposition \ref{prop:3Ddistr} can be strengthen in the following way.

\begin{prop}\label{prop:D1D3}
Assume that a geodesic $t\mapsto q(t)$ of metric $g$ is a horizontal curve of $\D_1$ or $\D_3$ with the initial data $(q(0),p(0))$ satisfying $q(0)\in\Omega$ and $p_2(0)>0$ in the first case, or $p_2(0)<0$ in the second case. Then the geodesic converges to $\Sigma$ at infinity (the hyperplane $q_1=q_2$, or $q_2=q_3$, respectively).
\end{prop}
\begin{proof}
Without loss of generality we limit ourselves to the case of $\D_1$. For this we assume $p_1=0$ and $p_2>0$ and our goal is to prove the $s(t)=q_1(t)-q_2(t)\to 0$ as $t\to\infty$. From \eqref{eq:hamilton} we have
\[
\dot s(t)=A(t)(e^{-s}(t)-1)
\]
with
\[
\begin{aligned}
A(t)&=p_2(t)+e^{q_3(t)-q_2(t)}p_3(t)=\\
&\frac{1}{2}\left(H_0\left(1+e^{q_3-q_2}\right)+\sqrt{(1-e^{q_3(t)-q_2(t)})(4H_1-(1+e^{q_3(t)-q_2(t)})H_0^2)}\right),
\end{aligned}
\]
where we computed $p_2$ and $p_3$ from formulas for $H_0$ and $H_1$ under assumption $p_1=0$. Moreover, $p_2>0$ implies that $A(t)>0$ for sufficiently large $t$ (and then $s(t)\to 0$ follows). Indeed,
\[
p_2=\frac{1}{2}H_0+\frac{1}{2}\sqrt{\frac{4}{1-e^{q_3(t)-q_2(t)}}H_1-\frac{1+e^{q_3(t)-q_2(t)}}{1-e^{q_3(t)-q_2(t)}}H_0^2}
\]
and from Section \ref{sec:2D} we know that $q_2(t)-q_3(t)\to \infty$ (since Theorem \ref{thm:estimmates2D} applies). Then $e^{q_3(t)-q_2(t)}\to 0$ and consequently $p_2(t)$ and $A(t)$ both converge to $\frac{1}{2}H_0+\frac{1}{2}\sqrt{4H_1-H_0^2}$ which  has to be positive since $p_2$ is positive by assumption.
\end{proof}

\subsection{Sectional curvature for 3-peakons}\label{sec:curvature3D}
In the sequel we shall need estimates for the sectional curvature.
First we compute the components of the (covariant) Riemann tensor of $(\Omega, g)$: $R_{ijkl}=g(R(\partial_k,\partial_l)\partial_j,\partial_i)$, where $R$ is the Riemann $(3,1)$-tensor. We skip the details of computations as they are lengthy and not illuminating.  As an outcome we get that all non-zero components are as follows up to the relations $R_{ijkl}=R_{klij}=-R_{jikl}$:
\[
R_{1212}=\frac{\left(3 {{e}^{3 q_3+2 q_2}}+2 {{e}^{2 q_3+ 3 q_2}}-2 {{e}^{q_3+4 q_2}}-2 {{e}^{5 q_2}}\right){{e}^{2 q_1}}+\left(-{{e}^{2 q_3+2q_2}}+{{e}^{q_3+3q_2}}+{{e}^{4 q_2}}\right){{e}^{3 q_1}}}{(e^{q_1}+e^{q_2})\Delta_1},
\]
\[
R_{2323}=\frac{\left(3 {{e}^{5 q_2}+2 {{e}^{4 q_2+q_1}}-2 {{e}^{3 q_2+2 q_1}}-2 {{e}^{2 q_2+3 q_1}}}\right)  {{e}^{2 q_3}}+\left( -{{e}^{5 q_2+q_1}}+{{e}^{4 q_2+2 q_1}}+{{e}^{3 q_2+3 q_1}}\right)  {{e}^{q_3}}}{(e^{q_2}+e^{q_3})\Delta_1},\]
\[
R_{1313}=\frac{{{e}^{q_3+3 q_2+2 q_1}}}{\Delta_1},\quad
R_{1213}=-\frac{{{e}^{2 q_3+2 q_2+2 q_1}}}{\Delta_1},\quad
R_{1223}=\frac{{{e}^{2 q_3+3 q_2+q_1}}}{\Delta_1},\quad
R_{1323}=-\frac{{{e}^{q_3+4 q_2+q_1}}}{\Delta_1},
\]
where
\[
\Delta_1=(e^{q_2}-e^{q_1})(e^{q_3}-e^{q_2})(e^{q_2}+e^{q_1})^2(e^{q_3}+e^{q_2})^2.
\]

Further, in order to compute sectional curvature $\kappa_\sigma$ of $\sigma$, a two-dimensional submanifold  of $(\Omega, g)$, whose tangent space is spanned by $X_a=(a_1,a_2,a_3)$ and $X_b=(b_1,b_2,b_3)$, we take
\begin{equation}\label{eq:seccurv1}
\kappa_\sigma=\frac{1}{g(X_a,X_a)g(X_b,X_b)-g(X_a,X_b)^2}\sum_{i,j,k,l=1}^3R_{ijkl}a_ia_kb_jb_l,
\end{equation}
where explicitly
\begin{equation}\label{eq:seccurv2}
\begin{aligned}
g(X_a,X_a)&g(X_b,X_b)-g(X_a,X_b)^2=\frac{{e}^{2(q_2+q_1)}}{\Delta_2} (2({a_1}{b_2}-{a_2} {b_1})({a_2}{b_3}-{a_3} {b_2}) {{e}^{q_3-q_1}}\\
&+2({a_1}{b_2}-{a_2} {b_1})({a_3}{b_1}-{a_1} {b_3}) {{e}^{q_3-q_2}}+2({a_2}{b_3}-{a_3} {b_2})({a_3}{b_1}-{a_1} {b_3}) {{e}^{q_2-q_1}}\\
&+(({a_1}{b_2}-{a_2} {b_1})^2+({a_2}{b_3}-{a_3} {b_2})^2+({a_3}{b_1}-{a_1} {b_3})^2)
\end{aligned}
\end{equation}
with
\[
\Delta_2=\left( {{e}^{q_2}}-{{e}^{q_1}}\right) \left( {{e}^{q_3}}-{{e}^{q_2}}\right) \left( {{e}^{q_2}}+{{e}^{q_1}}\right) \left( {{e}^{q_3}}+{{e}^{q_2}}\right).
\]
In what follows we shall use the known fact concerning the bound from above on the quotient of two quadratic forms. We give the proof for completeness and reader's convenience.
\begin{prop}\label{prop:prep}
Let $A$ and $B$ be symmetric $n\times n$ matrices. Moreover, assume $B$ to be positively defined (in particular invertible) matrix.
Take $\lambda_{max}$ the largest eigenvalue of $B^{-1}A$. Then for any $\zeta\neq 0$ we have
\[
\frac{\langle A\zeta, \zeta\rangle}{\langle B\zeta, \zeta\rangle}\leq \lambda_{max}\;.
\]
\end{prop}
\begin{proof}
We shall find the maximum of $\Phi(\zeta)=\frac{\langle A\zeta, \zeta\rangle}{\langle B\zeta, \zeta\rangle}$ on $\RR^n\setminus\{0\}$. Note that $\Phi$ is well defined in this domain because $B$ is positively defined. Moreover, the maximum exists in $\RR^n\setminus\{0\}$ as it equals the maximum of the function restricted to the unit sphere. In order to find $\zeta$ maximizing $\Phi$ we look for zeros of the gradient  $\nabla \Phi$. We get the following system of equations
\[
\sum_{j=0}^n\left(a_{ij}\zeta_j\langle B\zeta,\zeta\rangle-b_{ij}\zeta_j\langle A\zeta,\zeta\rangle\right)=0,\qquad i=1,\ldots,n,
\]
where $A=(a_{ij})_{i,j=1,\ldots,n}$ and $B=(b_{ij})_{i,j=1,\ldots,n}$. The system can be rewritten in the matrix form as
\[
A\zeta\langle B\zeta,\zeta\rangle-B\zeta\langle A\zeta,\zeta\rangle=0.
\]
Since $B$ is invertible we get
\[
B^{-1}A\zeta=\lambda\zeta
\]
with $\lambda=\frac{\langle A\zeta, \zeta\rangle}{\langle B\zeta, \zeta\rangle}$. It follows that $\zeta$ is an eigenvector of $B^{-1}A$ and the corresponding eigenvalue equals $\Phi(\zeta)$. Conversely, if $\zeta$ is an eigenvector corresponding to an eigenvalue $\lambda$ then $\lambda$ necessarily equals $\Phi(\zeta)$. Indeed
\[
\Phi(\zeta)=\frac{\langle A\zeta, \zeta\rangle}{\langle B\zeta, \zeta\rangle}=\frac{\langle BB^{-1}A\zeta, \zeta\rangle}{\langle B\zeta, \zeta\rangle}=\lambda\frac{\langle B\zeta, \zeta\rangle}{\langle B\zeta, \zeta\rangle}=\lambda.
\]
It follows that maximum of $\Phi$ is attained for $\zeta$ being an eigenvector of $B^{-1}A$ corresponding to the maximal eigenvalue.
\end{proof}

We are now in a position to state and prove a claim concerning the bound of sectional curvatures of $2$-dimensional submanifolds of $(\Omega,g)$, independent on the choice of vectors spanning them.
\begin{prop}\label{prop:curvature3D}
Consider any $2$-dimensional
subspace $\sigma$ of the tangent bundle of the 3-dimensional metric $g$. The sectional curvature $\kappa_\sigma$ satisfies
\[
\kappa_\sigma<\frac{1}{4}.
\]
Moreover $\kappa_\sigma$ tends to 0 as $q_1-q_2\to\infty$ and $q_2-q_3\to \infty$.
\end{prop}
\begin{proof}
Recall that we are in domain $\Omega$, i.e.~$q_1>q_2>q_3$.
Denote $Q_{ijkl}=R_{ijkl}\cdot\frac{\Delta_2}{{e}^{2(q_2+q_1)}}$. Direct computations show that
\[
Q_{1212}=\frac{3e^{3 q_3-q_2-2q_1}+2e^{2 q_3-2 q_1}-2e^{q_3+q_2-2q_1}-2 e^{2q_2-2q_1}-e^{2q_3-q_2-q_1}+e^{q_3-q_1}+e^{q_2-q_1}}{(1+e^{q_2-q_1})\Delta_3},
\]
\[
Q_{2323}=\frac{3e^{2 q_3+q_2-3q_1}+2e^{2 q_3-2 q_1}-2e^{2q_3-q_2-q_1}-2 e^{2q_3-2q_2}-e^{q_3+q_2-2q_1}+e^{q_3-q_1}+e^{q_3-q_2}}{(1+e^{q_3-q_2})\Delta_3},\]
\[
Q_{1313}=\frac{{{e}^{q_3-q_1}}}{\Delta_3},\quad
Q_{1213}=-\frac{{{e}^{2 q_3 - q_2-q_1}}}{\Delta_3},\quad
Q_{1223}=\frac{{{e}^{2 q_3-2q_1}}}{\Delta_3},\quad
Q_{1323}=-\frac{{{e}^{q_3+q_2-2q_1}}}{\Delta_3},
\]
where
\[
\Delta_3=(1+e^{q_2-q_1})(1+e^{q_3-q_2}).
\]

It follows that the singularity of $R_{ijkl}$, which is of the form $\left( {{e}^{q_2}}-{{e}^{q_1}}\right)^{-1} \left( {{e}^{q_3}}-{{e}^{q_2}}\right)^{-1}$, disappears in $Q_{ijkl}$. Moreover, all coefficients $Q_{ijkl}$ decay when when $q_1-q_2\to\infty$ and $q_2-q_3\to \infty$. Now, using the relations $R_{ijkl}=R_{klij}=-R_{jikl}$ one gets the formula
\begin{equation}\label{eq:seccurv3}
\begin{aligned}
\frac{\Delta_2}{{e}^{2(q_2+q_1)}}\sum_{i,j,k,l=1}^3&R_{ijkl}a_ia_kb_jb_l=\,2({a_1}{b_2}-{a_2} {b_1})({a_2}{b_3}-{a_3} {b_2}) Q_{1223}\\
&-2({a_1}{b_2}-{a_2} {b_1})({a_3}{b_1}-{a_1} {b_3}) Q_{1213}-2({a_2}{b_3}-{a_3} {b_2})({a_3}{b_1}-{a_1} {b_3}) Q_{1323}\\
&+({a_1}{b_2}-{a_2} {b_1})^2Q_{1212}+({a_2}{b_3}-{a_3} {b_2})^2Q_{2323}+({a_3}{b_1}-{a_1} {b_3})^2Q_{1313}
\end{aligned}
\end{equation}
Hence, denoting
\[
\zeta_1={a_2}{b_3}-{a_3} {b_2},\qquad \zeta_2={a_3}{b_1}-{a_1} {b_3},\qquad \zeta_3={a_1}{b_2}-{a_2} {b_1}
\]
and substituting $\zeta=(\zeta_1, \zeta_2, \zeta_3)$ in \eqref{eq:seccurv2} and \eqref{eq:seccurv3}, we get that \eqref{eq:seccurv1} is a ratio of two bilinear forms in $\zeta$. Precisely
\[
\kappa_\sigma=\frac{\langle Q\zeta,\zeta\rangle}{\langle E\zeta,\zeta\rangle},
\]
where, as before, $E=(e^{-|q_i-q_j|})$ is the inverse matrix of metric $g$ and
\[
Q=\left[
	\begin{array}{ccc}
	 Q_{2323} & -Q_{1323} & Q_{1223}\\
	 -Q_{1323} & Q_{1313} & -Q_{1213}\\
	 Q_{1223}  & -Q_{1213} & Q_{1212}
	\end{array}
\right].
\]
Applying Proposition \ref{prop:prep}, we infer that $\frac{\langle Q\zeta,\zeta\rangle}{\langle E\zeta,\zeta\rangle}$ attains its maximum for a vector $\zeta$ being an eigenvector of $E^{-1}Q$ and the maximum equals the maximal eigenvalue of $E^{-1}Q$. Indeed, we know that matrix $E$ is positively defined in $\Omega$, see \cite{CGKM} for instance, so that assumptions of Proposition \ref{prop:prep} are satisfied. Computation shows that there are three eigenvalues of $E^{-1}Q$
\[
\begin{aligned}
&\lambda_1=\frac{e^{q_3-q_2}-2e^{2q_3-2q_2}}{(1+e^{q_3-q_2})^2},\\
&\lambda_2=\frac{e^{q_3-q_1}}{(1+e^{q_2-q_1})(1+e^{q_3-q_2})},\\
&\lambda_3=\frac{e^{q_2-q_1}-2e^{2q_2-2q_1}}{(1+e^{q_2-q_1})^2}.
\end{aligned}
\]
All the eigenvalues are bounded from above by $\frac{1}{4}$. Indeed, $\lambda_1$ and $\lambda_3$ are both functions of one variable $t\in (-\infty, 0)$ of the form $\frac{e^t-2e^{2t}}{\left(1+e^t\right)^2}$. One notices that the maximal value of such a function equals $1/12$. In the case of $\lambda_2$ we estimate
\[
\lambda_2=\frac{e^{q_3-q_1}}{1+e^{q_3-q_2}+e^{q_2-q_1}+e^{q_3-q_1}}=\frac{1}{e^{q_1-q_3}+e^{q_1-q_2}+e^{q_2-q_3}+1}<1/4,
\]
the last inequality is clear, since $q_1,q_2,q_3\in \Omega$.
\end{proof}
\begin{remark}
The three eigenvalues $\lambda_1$, $\lambda_2$, $\lambda_3$ correspond to three planes $\D_1$, $\D_2$ and $\D_3$, respectively. Note that the curvature on $\D_1$ and $\D_3$ coincides with the Gaussian curvature in the 2-dimensional case.
\end{remark}

\subsection{Collisions of 3-peakons}
As an application of Proposition \ref{prop:curvature3D} we give a new geometric proof of the necessary condition for a collision.
\begin{theorem}\label{thm:main3D}
Let $u(x,t)=\sum_{i=1}^3p_i(t)e^{-|x-q_i(t)|}$ be a threepeakon solution to the Camassa--Holm equation with initial data $(q(0),p(0))$ satisfying $q(0)\in\Omega$ and $p_i(0)\neq 0$, $i=1,2,3$.  Then the necessary condition for the threepeakon to collide is
\begin{equation}\label{eq:cond3D1}
p_1(0)<0<p_2(0),
\end{equation}
or
\begin{equation}\label{eq:cond3D2}
p_2(0)<0<p_3(0).
\end{equation}
\end{theorem}
\begin{proof}
We shall prove that condition \eqref{eq:cond3D1} is necessary for a collision of $q_1$ with $q_2$. Analogously, condition \eqref{eq:cond3D2} is necessary for a collision $q_2=q_3$. Consequently, if neither \eqref{eq:cond3D1} nor \eqref{eq:cond3D2} hold, then there is no collision. We focus on proving necessity of \eqref{eq:cond3D1}, the other case is parallel.

Let $q^0\in\Omega$ be the initial point of a geodesic $t\mapsto q(t)$, i.e. $q(0)=q^0$. Our aim is to prove that if $p_1>0$ or $p_2<0$, then $q(t)$ never hits the halfplane $\Sigma_1=\{q_1=q_2,\ q_3<q_1\}$, which is the singular set in question. We consider two submanifolds of $\Omega$ that pass through $q^0$. $N_1$ is the integral leaf of $\D_1$ and $N_2$ is defined as $\exp_{q^0}(\D_2)$, where $\exp_{q^0}\colon T_{q^0}\Omega\to\Omega$ is the exponential mapping of metric $g$, that sends a tangent vector $V\in T_{q^0}\Omega$ to a point $\gamma(1)$, where $\gamma$ is the unique geodesic starting at $q^0$ with the tangent vector $V$. Note that $\exp_{q^0}(\D_1)$ coincides with $N_1$, since leaves of $\D_1$ are totally geodesic with respect to the metric $g$.

Now, according to Proposition \ref{prop:3Ddistr}, $N_1$ splits $\Omega$ into two parts and if $p_1>0$, then the tangent vector to the geodesic $t\mapsto q(t)$ is directed into the sector that is separated from $\Sigma_1$ by $N_1$.

Similarly, $N_2$ splits $\Omega$ into two parts. If $p_2=0$ then a threepeakon is actually a twopeakon and, due Theorem \ref{thm:collisions2D}, horizontal geodesics of $\D_2$ hit the singularity $\{q_1=q_3\}$ at  a finite time. Moreover, Lemma \ref{lemma:D2} implies that  $N_2$ joins $q^0$ with the boundary $\partial\Sigma_1=\{q_1=q_2=q_3\}$. Again, if $p_2<0$, then the tangent vector to the geodesic $t\mapsto q(t)$ is directed into the sector that is separated from $\Sigma_1$ by $N_2$. Hence, in both cases, if $t\mapsto q(t)$ hits $\Sigma_1$ at certain time $t^*$ (a collision time), then before it happens, say at $t_1<t^*$, it must hit $N_1$ or $N_2$, respectively. By definition of $N_i$'s, $t_1$ would be a conjugate time for the geodesic. Moreover, at time $t_1$ either $p_1>0$ or $p_2<0$, as at the initial point, because the conditions $p_i=0$ are preserved by the geodesic flow of $g$, and consequently $p_i$'s cannot change signs. Therefore, the reasoning can be repeated and we get a sequence of conjugate times $t_1<t_2<t_3<\ldots<t^*$. However, there is a lower bound on the difference $t_{i+1}-t_i$ that follows from the Rauch comparison Theorem \ref{thm:rauch1}, because there is an upper bound on the sectional curvature of $g$ (Proposition \ref{prop:curvature3D}). This implies that $t^*$ cannot be finite.

\end{proof}

\noindent
{\bf Acknowledgement.} T.C. was partially supported by the National Science Centre grant
SONATA BIS 7 number UMO-2017/26/E/ST1/00989. W.K. was supported by the grant 2019/34/E/ST1/00188 from the National Science Centre, Poland.

\end{document}